\documentclass[11pt]{amsart}
\usepackage{amssymb}
\usepackage{amsmath}
\usepackage{graphicx}
\usepackage{hyperref}
\usepackage{xcolor}
\usepackage{soul}
\usepackage{tikz}

\newtheorem{thm}{Theorem}[section]
\newtheorem{lem}[thm]{Lemma}
\newtheorem{fact}[thm]{Fact}
\newtheorem{cor}[thm]{Corollary}

\newtheorem{problem}[thm]{Problem}
\newtheorem{Def}[thm]{Definition}
\newtheorem{prop}[thm]{Proposition}
\newtheorem{rem}[thm]{Remark}
\newtheorem{ex}[thm]{Example}

\newcommand{\bdfn}{\begin{Def} \rm}
\newcommand{\edfn}{\end{Def}}

\newcommand{\cross}{\mathbin{\tikz [x=1.4ex,y=1.4ex,line width=.2ex] \draw (0,0)--(1,1) (0,1)--(1,0);}}

\newcommand{\beqa}{\begin{eqnarray*}}
\newcommand{\eeqa}{\end{eqnarray*}}

\newcounter{cnt1}
\newcounter{cnt2}
\newcounter{cnt3}
\newcounter{cnt4}
\newcommand{\blr}{\begin{list}{$($\roman{cnt1}$)$} {\usecounter{cnt1}
 \setlength{\topsep}{0pt} \setlength{\itemsep}{0pt}}}
\newcommand{\blR}{\begin{list}{\Roman{cnt4}.\ } {\usecounter{cnt4}
 \setlength{\topsep}{0pt} \setlength{\itemsep}{0pt}}}
\newcommand{\bla}{\begin{list}{$(\alph{cnt2})$} {\usecounter{cnt2}
 \setlength{\topsep}{0pt} \setlength{\itemsep}{0pt}}}
\newcommand{\bln}{\begin{list}{$($\arabic{cnt3}$)$} {\usecounter{cnt3}
 \setlength{\topsep}{0pt} \setlength{\itemsep}{0pt}}}
\newcommand{\el}{\end{list}}

\begin{document}

\title[\tiny{weak$^*$-weak points of continuity on the state spaces }]{weak$^*$-weak points of continuity on the state spaces}

\author{Saurabh Dwivedi}

\address{Shiv Nadar Institution of Eminence. Gautam Buddha Nagar Delhi NCR, Uttar Pradesh-201314, India}

\email{sd605@snu.edu.in, saurabhdewedi876@gmail.com}

\subjclass[2020]{Primary 46A22, 46B10, 46B25; Secondary 46B20, 46B22. \hfill
\textbf{\today}}

\keywords{Smooth points, State spaces, Weak compactness, \textnormal{w}$^*$-\textnormal{w} PC, Hölder's inequality}
\begin{abstract}
	Let $X$ be Banach space. For $x\in X$ with $\|x\|=1$, we denote the state space by, $S_{x}=\{x^*\in X^*:\|x^*\|=x^*(x)=1\}$. In this paper, we study weak$^*$-weak and weak$^*$-$\|.\|$ points of continuity of the identity map, on the state spaces in the space $\ell^{p}(X)$ for $1<p<\infty$ for a non-reflexive Banach space $X$ and then we use these results to characterize the weak and norm compactness of the state spaces of unit vectors in $\ell^{p}(X)$. In addition, we address an open problem, the characterization of weakly compact state spaces in the space of Bochner-integrable functions $L^{1}(\mu, X)$ (see \cite[Problem~3.18]{SDD}). We also provide a local solution to this problem without any additional assumptions on the Banach space $X$. Motivated by the work of \cite{SMR}, we show that if the set of all weak$^{*}$-weak points of continuity of $L^{1}(\mu, X)_{1}^{*}$ is weakly dense in $L^{1}(\mu, X)_{1}^{*}$, then $X^{*}$ has the Radon–Nikodým property (RNP) (see Theorem~\ref{T312}).

\end{abstract}
\maketitle
\textbf{To appear in Revista de la Real Academia de Ciencias Exactas, Físicas y Naturales. Serie A. Matemáticas.}
\section{Introduction}
Let $X$ be a real or complex Banach space, and let $X^*$ be the dual space of $X$. We denote by $X_1$, the unit ball of $X$. Let $S(X)$ be the unit sphere of $X$. For $x\in S(X)$, let $S_x=\{x^*\in S(X^*):x^*(x)=1\}$ be the \textbf{state space} associated to $x$. By the Hahn-Banach theorem, $S_x$ is always non-empty. Let $K\subseteq X_{1}^*$ be a norm closed and bounded subset of $X_{1}^*$. We recall from \cite{HL93} that a point $x^*\in K$ is said to be a \textbf{\textnormal{w}$^*$-PC} if the identity map $I:(K,\textnormal{w}^*)\rightarrow(K,\|.\|)$ is continuous at $x^*$. Motivated by this, we define that a point $x^*\in K$ is said to be a \textbf{\textnormal{w}$^*$-\textnormal{w} PC} if the identity map $I:(K,\textnormal{w}^*)\rightarrow(K,\textnormal{w})$ is continuous at $x^*$. When $K=S(X^*)$, we refer the reader to \cite{SMR} and \cite{SD} for a detailed study of w$^*$-w PC's of $K$. Here the authors have shown that when X is not a reflexive space, w$^*$-w PC's of $X_{1}^*$ must belong to $S(X^*)$. Classification of certain Banach spaces for which state spaces are weakly compact sets was carried out in \cite{Rao}. If $S_{x}$ is a weakly compact set, then it is easy to see that w-topology and w$^*$-topology coincide on $S_{x}$, and thus $I:(S_{x}, \textnormal{w}^*)\to (S_{x}, \textnormal{w})$ is a continuous map. Conversely, if $I:(S_{x}, \textnormal{w}^*)\to (S_{x}, \textnormal{w})$ is a w$^*$-w continuous map and since $S_{x}$ is always a w$^*$-compact set, it is a weakly compact set. Thus, it is an interesting question to determine points of continuity of the map $I:(S_{x}, \textnormal{w}^*)\to (S_{x}, \textnormal{w})$. Same remark also apply in the case of norm compactness and w$^*$-PC's.
 In our recent work (see, \cite{SDD}), we gave a complete characterization of weakly compact state spaces in some spaces of vector-valued functions. Let $X$ be Banach space and let $1\leq p<\infty$. We define $\ell^{p}(X)=\{(x_{i})_{i\in \mathbb{N}}:x_{i}\in X \text{ and }\sum_{i=1}^{\infty}\|x_{i}\|^{p}<\infty\}$, equipped with the $\ell^{p}$-norm. We now recall from \cite{SDD}:
\begin{thm}\cite[Theorem~3.14]{SDD}\label{3t12}
    Let $ (x_i) \in {\ell^{1}}(X)$ be such that $ \| (x_i) \| = 1 $, and let $ x_{i}\neq 0$ for all $i\in{\mathbb{N}}$. Then $ S_{(x_i)}$ is weakly compact in $ {{\ell^{\infty}}}(X_{}^*) $ if and only if $ S_{\frac{x_i}{\|x_i\|}} $ is weakly compact in $X^*$ for each $i\in{\mathbb{N}}$ and $\textup{diam}\left(S_{\frac{x_i}{\|x_i\|}}\right) \to 0 $.
\end{thm}

\begin{prop}\cite[Proposition~3.16]{SDD}\label{3c4}
Let $ (x_i) \in {{\ell^{1}}}(X_{})$ with $ \| (x_i) \| = 1 $, and $ x_{i}\neq 0$ for all $i\in{\mathbb{N}}$. Then $S_{(x_i)}$ is norm compact in $ {{\ell^{\infty}}}(X_{}^*) $ if and only if $S_{\frac{x_i}{\|x_i\|}} $ is norm compact in $X_{}^*$ for each $i\in{\mathbb{N}}$ and $ \textup{diam}\left(S_{\frac{x_i}{\|x_i\|}}\right) \to 0 $.
\end{prop}

In the present article, we provide a characterization of weakly compact and norm compact state spaces in the space $\ell^{p}(X)$ for $1<p<\infty$, when $X$ is a non-reflexive Banach space. Interestingly, in this case, it is not required that the diameters of the coordinatewise state spaces tend to $0$. This marks an important distinction between the cases $p=1$ and $1<p<\infty$ (see Corollary~2.6 and Corollary~2.9). 

Let $(\Omega, \Sigma, \mu)$ be a measure space, where $\mu$ is a non-atomic probability measure. We use $L^p(\mu, X)$ to denote the space of all (equivalent classes of) $X$-valued \textbf{Bochner $p$-integrable} functions, equipped with the norm $\|f\|=\left(\int_{\Omega}\|f(\omega)\|^{p}d(\mu(\omega))\right)^{\frac{1}{p}}$ for $f\in L^p(\mu, X)$. We refer the reader to \cite{DU} for more information on this topic. Let $\mu$ be a finite measure and suppose that $X^*$ has the Radon-Nikodym property (RNP) with respect to $\mu$. By \cite[Theorem~IV.1.1]{DU}, we then have $L^{p}(\mu, X)^{*} = L^{q}(\mu, X^{*})$, where $1<p\leq q<\infty$ are such that $\frac{1}{p}+\frac{1}{q}=1$ and $L^1(\mu, X)^*=L^{\infty}(\mu, X^*)$, the space of all $X^*$-valued strongly measurable functions that are essentially bounded. More generally, if $X$ is any Banach space, we still have the canonical embedding $L^{q}(\mu, X^*)\subseteq L^p(\mu, X)^*$. Let $A \subseteq \Omega$ be a measurable subset of positive measure, and set $\Sigma' = \Sigma|_{A}$ (the restricted $\sigma$-algebra) and $\lambda = \mu|_{A}$.  
If $X$ has the RNP with respect to $\mu$, then it is easy to see that $X$ also has the RNP with respect to $\lambda$.
  Let $K\subseteq L^{q}(\mu, X^*)$. Then note that $K$ is weakly compact in $L^{q}(\mu, X^*)$ if and only if it is weakly compact in $L^p(\mu, X)^*$. Let $\Omega = [0,1]$, and let $\mu$ denote the Lebesgue measure on the Lebesgue $\sigma$-algebra of $[0,1]$.
 The following theorem established in \cite{SDD}, characterizes norm compactness of the state spaces in $L^{1}([0,1], X)$.
  
	\begin{thm}\label{Th316}
    Assume $X^*$ is separable. Let $f\in L^{1}([0,1],X)$ be such that $\|f\|=1$, and $f(t)\neq0$ a.e. Then $S_f$ is norm compact if and only if $f$ is smooth.
\end{thm}

We address the following problem that remains unanswered in \cite{SDD}.

\begin{problem}\cite[Problem~3.18]{SDD}\label{P1}
Let $f \in S(L^{1}([0,1],X))$. It is not known under what conditions the set $S_{f}$ is weakly compact.

\end{problem}

In this paper, we provide a complete answer to this question. More precisely, let $f\in S(L^{1}(\mu, X))$. We will characterize the weak compactness of the state space $S_{f}$ under the assumption that $S_{f}\cap L^{\infty}(\mu, X^*)$ is non-empty (see, Theorem~\ref{T36}).

Let $X$ be a space such that $X=Y\bigoplus_{\ell^{p}}Z$ ($\ell^{p}$-sum, $1\leq p<\infty$) for some closed subsapce $Z$ of $X$. By duality, we can write $X^*=Y^{*}\bigoplus_{\ell^{q}}Z^{*}$. Let $(y, 0)\in S(X)$. We denote by $S_{(y, 0)}^{Y^*}$, the set $\{(y^*, 0)\in X_{1}^*:(y^*, 0)(y, 0)=1\}$. We refer to the set $S_{(y, 0)}^{Y^{*}}$ as the state space of $(y, 0)$ in $Y^{*}$. In section~1, we start with Proposition~\ref{P2}, where we will show that a point $(y^*, 0)\in S_{(y, 0)}^{Y^*}$ is a w$^*$-w PC if and only if it is a w$^*$-w PC in $S_{(y, 0)}$. We present our main results as Theorem~\ref{T4} and Theorem~\ref{T4'}, where we will characterize w$^*$-w and w$^*$-$\|.\|$ PC's of the state spaces in the space $\ell^{p}(X)$. Finally, we conclude this section with two corollaries to our main theorems, namely, Theorem~\ref{T4} and Theorem~\ref{T4'} concerning the weak compactness and norm compactness of the state spaces in $\ell^{p}(X)$. In section~2, we address Problem~\ref{P1} and provide a complete characterization for weakly compact state spaces in the space $L^{1}(\mu, X)$. Let $A$ be a set, we denote by $A^c$, the set-theoretic complement of $A$.

 \section{On the state spaces of $\ell^{p}(X)$}
Let us begin with the following proposition. Throughout this section, we assume $1 < p < \infty$, and let $q$ denote the conjugate exponent of $p$, that is, $\frac{1}{p}+\frac{1}{q}=1$.
\begin{prop}\label{P2}
   Let $Y\subseteq X$ such that $X=Y\bigoplus_{\ell^{p}}Z$. If $(y, 0)\in S(X)$, then $S_{(y, 0)}^{Y^*}=S_{(y, 0)}$. In particular, $(y^*, 0)\in S_{(y, 0)}^{Y^*}$ is a \textnormal{w}$^*$-\textnormal{w} PC if and only if $(y^*, 0)\in S_{(y, 0)}$ is a \textnormal{w}$^*$-\textnormal{w} PC. 
\end{prop}
\begin{proof}
   Since $X=Y\bigoplus_{\ell^{p}}Z$, by duality, we can write $X^*=Y^{*}\bigoplus_{\ell^{q}}Z^{*}$. Suppose $x^*\in S_{(y, 0)}$ such that $x^*=(y^*, z^*$), where $y^*\in Y^{*}, z^*\in Z^{*}$ and $1=\|x^*\|^{q}=\|y^*\|^{q}+\|z^*\|^{q}.$ We have $(y^*, z^*)(y, 0)=y^*(y)=1.$ Since $\|y^*\|\leq 1$, it gives that $\|y^*\|=1.$ This implies that $\|z^*\|=0$. So we get $x^*=(y^*, 0)$ and $(y^*, 0)(y, 0)=1$. Thus, $S_{(y, 0)}\subseteq S_{(y, 0)}^{Y^*}$. It is easy to check the other inclusion. This completes the proof.
\end{proof}

The following corollary is an immediate consequence of the above proposition. It relates to the norm and weak compactness of the state spaces.
\begin{cor}\label{R2}
    Let $X$ be such that $X=Y\bigoplus_{\ell^{p}}Z$. We have that $S_{(y, 0)}^{Y^*}$ is weakly compact in $Y^*$ if and only if $S_{(y, 0)}$ is weakly compact in $X^*$. Similarly, we have $S_{(y, 0)}^{Y^*}$ is norm compact in $Y^*$ if and only if $S_{(y, 0)}$ is norm compact in $X^*$ 
\end{cor}

Let $\Theta:\ell^{q}(X^*) \to \ell^{q}$ be the map defined by $\Theta\big((x_{i}^*)\big) = \big(\|x_{i}^*\|\big).$ We now present a lemma showing that the image of a state space under the map $\Theta$ is a singleton. This result plays a crucial role throughout the article. Before stating the lemma, we recall the \textbf{Hölder's inequality}.
 Let $(\alpha_{i})\in \ell^{p}$ and $(\beta_{i})\in \ell^{q}$, we have that $(\alpha_{i}\beta_{i})\in \ell^{1}$ and $\sum_{i=1}^{\infty}|\alpha_{i}\beta_{i}|\leq \|(\alpha_{i})\|_{p}\|(\beta_{i})\|_{q}$. The equality holds true if and only if there exists a scalar $\lambda\in\mathbb{C}$ such that $|\alpha_{i}|^{p}=\lambda|\beta_{i}|^{q}$ for each $i\in \mathbb{N}$.
\begin{lem}\label{L3}
     Let $(x_{i})\in S(\ell^{p}(X))$ and let $(x_{i}^*), (y_{i}^*)\in S_{(x_{i})}$. Then $\|x_{i}^*\|=\|y_{i}^*\|$ for each $i\in \mathbb{N}$.
\end{lem}
\begin{proof}
    Since $(x_{i}^*)\in S_{(x_{i})}$, so we have 
    \begin{align*}
    1=(x_{i}^*)((x_{i}))&=\sum_{i=1}^{\infty}x_{i}^*(x_{i})\\
    &\leq\sum_{i=1}^{\infty}|x_{i}^*(x_{i})|\\
    &\leq\sum_{i=1}^{\infty}\|x_{i}^*\|\|x_{i}\|.\\
    \end{align*}
    Since $(\|x_{i}^*\|)\in S(\ell^{q})$ and $(\|x_{i}\|)\in S(\ell^{p})$, by the Hölder's inequality, we have that 
    \begin{align*}
    1&\leq\sum_{i=1}^{\infty}\|x_{i}^*\|\|x_{i}\|\\
    &\leq(\sum_{i=1}^{\infty}\|x_{i}^*\|^{q})^{\frac{1}{q}}(\sum_{i=1}^{\infty}\|x_{i}\|^{p})^{\frac{1}{p}}\\
    &=1.
    \end{align*}
    This gives that 
    \[
    \sum_{i=1}^{\infty}\|x_{i}^*\|\|x_{i}\|=(\sum_{i=1}^{\infty}\|x_{i}^*\|^{q})^{\frac{1}{q}}(\sum_{i=1}^{\infty}\|x_{i}\|^{p})^{\frac{1}{p}}=1.
    \]
    Thus equality holds in Hölder's inequality. Consequently, the sequences $(\|x_{i}^*\|^{q})$ and $(\|x_{i}\|^{p})$
are linearly dependent. So there exists a $\lambda\in\mathbb{C}$ such that $(\|x_{i}^*\|^{q})=\lambda(\|x_{i}\|^{p})$. Consequently, $\|x_{i}^*\|^{q}=\lambda\|x_{i}\|^{p}$ for each $i\in \mathbb{N}$. Since $\|(x_{i}^*)\|_{q}=\|(x_{i})\|_{p}=1$, we get that $\lambda=1$. Thus $\|x_{i}^*\|^{q}=\|x_{i}\|^{p}$. Similar argument when applies to $(y_{i}^*)$ gives that $\|y_{i}^*\|^{q}=\|x_{i}\|^{p}$. Combining the last two equations, we get that $\|x_{i}^*\|^{q}=\|y_{i}^*\|^{q}$, and hence, $\|x_{i}^*\|=\|y_{i}^*\|$ for each $i\in \mathbb{N}$. 
\end{proof}
Let $(x_{i}^{*}) \in S_{(x_{i})}$. The following remark relates the normalized coordinates of $(x_{i}^{*})$ to the coordinatewise state spaces of the elements $x_{i}$.

\begin{rem}\label{R24}
From the proof of the above lemma, we observe that $Re(x_{i}^*(x_{i}))=\|x_{i}^*\|\|x_{i}\|$ and $|x_{i}^*(x_{i})|=\|x_{i}^*\|\|x_{i}\|$. Therefore, $x_{i}^*(x_{i})=\|x_{i}^*\|\|x_{i}\|$ for each $i\in\mathbb{N}$. This gives that $\frac{x_{i}^*}{\|x_{i}^*\|}\in S_{\frac{x_{i}}{\|x_{i}\|}}$ for each $i\in \mathbb{N}$, where $\|x_{i}^*\|\neq 0$.	
\end{rem}
In the next theorem, we make use of Lemma~\ref{L3} to characterize w$^*$-w PC's of the state spaces in the space $\ell^{p}(X)$. 
\begin{thm}\label{T4}
    Let $(x_{i})\in S(\ell^{p}(X))$ with $x_{i}\neq 0$ for each $i\in \mathbb{N}$. Then $I:(S_{(x_{i})}, \textnormal{w}^*)\to (S_{(x_{i})}, \textnormal{w})$ is continuous at a point $(x_{i}^*)$ if and only if for each $i\in \mathbb{N}$, the functional $\frac{x_{i}^*}{\|x_{i}^*\|}\in S_{\frac{x_{i}}{\|x_{i}\|}}$ is a \textnormal{w}$^*$-\textnormal{w} PC.
\end{thm}
\begin{proof}
     Let $(x_{i}^*)\in S_{(x_{i})}$ be a w$^*$-w PC and fix $i_{0}\in \mathbb{N}$. To show that the functional $\frac{x_{i_{0}}^*}{\|x_{i_{0}}^*\|}\in S_{\frac{x_{i_{0}}}{\|x_{i_{0}}\|}}$ is a w$^*$-w PC, we consider a net $\{x_{\alpha}^*\}_{\alpha\in J}\in S_{\frac{x_{i_{0}}}{\|x_{i_{0}}\|}}$ such that $x_{\alpha}^*\to \frac{x_{i_{0}}^*}{\|x_{i_{0}}^*\|}$ in the w$^*$-topology. For $\alpha\in J$, we define
     \[x_{i}^*(\alpha)=
     \begin{cases}
        x_{\alpha}^*\|x_{i}^*\|, &\text{ if } i=i_{0},\\
        x_{i}^*, & \text{ otherwise }.
     \end{cases}
     \]
     So for each $\alpha\in J$, we get that
     \begin{align*}
		(x_{i}^*(\alpha))((x_{i}))&=\left(\sum_{i\neq i_{0}}x_{i}^*(x_{i})\right)+\|x_{i_{0}}^*\|x_{\alpha}^*(x_{i_{0}})\\
       (x_{i}^*(\alpha))((x_{i}))&=\left(\sum_{i\neq i_{0}}\|x_{i}^*\|\|x_{i}\|\right)+\|x_{i_{0}}^*\|\|x_{i_{0}}\|\\
       &=\sum_{i=1}^{\infty}\|x_{i}^*\|\|x_{i}\|\\
       &=1. 
     \end{align*} 
     Also, note that for each $\alpha\in J$, we have
     $$
\|(x_{i}^*(\alpha))\|^{q}_{q}=\sum_{i=1}^{\infty}\|x_{i}^*\|^{q}=1.
     $$
     So we get a net, namely, $\{(x_{i}^*(\alpha))\}_{\alpha\in J}\subseteq S_{(x_{i})}$ such that $(x_{i}^*(\alpha))\to (x_{i}^*)$ in the w$^*$-topology. By the hypothesis, we have that $(x_{i}^*(\alpha))\to (x_{i}^*)$ in the w-topology. So we get the cordinatewise weak convergence which implies that $x_{\alpha}^*\|x_{i_{0}}^*\|\to x_{i_{0}}^*$ in the w-topology. Hence, $x_{\alpha}^*\to \frac{x_{i_{0}}^*}{\|x_{i_{0}}^*\|}$ in the w-topology. This proves the necessary part.

     Conversely, suppose for each $i\in \mathbb{N}$, the functional $\frac{x_{i}^*}{\|x_{i}^*\|}\in S_{\frac{x_{i}}{\|x_{i}\|}}$ is a w$^*$-w PC. Let $\{(x_{i}^*(\alpha))\}_{\alpha\in J}\subseteq S_{(x_{i})}$ be a net such that $(x_{i}^*(\alpha))\to (x_{i}^*)$ in the w$^*$-topology. From Lemma~\ref{L3}, we have that $\|x_{i}^*(\alpha)\|=\|x_{i}^*\|$ for each $i\in\mathbb{N}$ and for each $\alpha\in J$. Suppose $(x_{i}^{**})\in \ell^{p}(X^{**})$ and $\epsilon>0$. By the virtue of the Hölder's inequality, we can choose $M\in \mathbb{N}$ such that $\sum_{i=M+1}^{\infty}\|x_{i}^{*}\|\|x_{i}^{**}\|<\frac{\epsilon}{4}.$
     
     Now we claim that there exists $\alpha_{0}\in J$ such that $|\sum_{i=1}^{\infty}x_{i}^{**}(x_{i}^*(\alpha)-x_{i}^*)|<\epsilon$ for all $\alpha\geq \alpha_{0}$. We have 
     \begin{align*}
        |\sum_{i=1}^{\infty}x_{i}^{**}(x_{i}^*(\alpha)-x_{i}^*)|&=|\sum_{i=1}^{\infty}\|x_{i}^*\|x_{i}^{**}\left(\frac{x_{i}^*(\alpha)}{\|x_{i}^*(\alpha)\|}-\frac{x_{i}^*}{\|x_{i}^*\|}\right)|\\
        &\leq\sum_{i=1}^{M}\|x_{i}^*\||x_{i}^{**}\left(\frac{x_{i}^*(\alpha)}{\|x_{i}^*(\alpha)\|}-\frac{x_{i}^*}{\|x_{i}^*\|}\right)|\\
        &+\sum_{i=M+1}^{\infty}\|x_{i}^*\|\|x_{i}^{**}\|\|\left(\frac{x_{i}^*(\alpha)}{\|x_{i}^*(\alpha)\|}-\frac{x_{i}^*}{\|x_{i}^*\|}\right)\|.\\
\end{align*}
Since $\|\left(\frac{x_{i}^*(\alpha)}{\|x_{i}^*(\alpha)\|}-\frac{x_{i}^*}{\|x_{i}^*\|}\right)\|\leq 2$, we get that 
        \begin{align}
            |\sum_{i=1}^{\infty}x_{i}^{**}(x_{i}^*(\alpha)-x_{i}^*)|\leq\sum_{i=1}^{M}\|x_{i}^*\||x_{i}^{**}\left(\frac{x_{i}^*(\alpha)}{\|x_{i}^*(\alpha)\|}-\frac{x_{i}^*}{\|x_{i}^*\|}\right)|+\frac{\epsilon}{2}.
        \end{align}
        By our assumption, $(x_{i}^*(\alpha))\to (x_{i}^*)$ in the w$^*$-topology. This implies that for each $i\in\mathbb{N}$, we have $x_{i}^*(\alpha)\to x_{i}^*$ in the w$^*$-topology. As $\|x_{i}^*(\alpha)\|=\|x_{i}^*\|$ for each $i\in\mathbb{N}$ and for each $\alpha\in J$ , we get that $\frac{x_{i}^*(\alpha)}{\|x_{i}^*(\alpha)\|}\to \frac{x_{i}^*}{\|x_{i}^*\|}$ in the w$^*$-topology. By the hypothesis, we have $\frac{x_{i}^*(\alpha)}{\|x_{i}^*(\alpha)\|}\to \frac{x_{i}^*}{\|x_{i}^*\|}$ in the w-topology. We choose $\alpha_{0}\in J$ such that for each $1\leq i\leq M$, we have 
        \begin{align}
        |x_{i}^{**}\left(\frac{x_{i}^*(\alpha)}{\|x_{i}^*(\alpha)\|}-\frac{x_{i}^*}{\|x_{i}^*\|}\right)|<\frac{\epsilon}{2M\|x_{i}^*\|} \text{ for all } \alpha\geq \alpha_{0}.
    \end{align}
     By Combining $(1)$ and $(2)$, we get that $|\sum_{i=1}^{\infty}x_{i}^{**}(x_{i}^*(\alpha)-x_{i}^*)|<\epsilon$ for all $\alpha\geq \alpha_{0}$. This establishes the claim. Hence $(x_{i}^*(\alpha))\to (x_{i}^*)$ in the w-topology. This completes the proof.
\end{proof}
In the next corollary, we discuss the case, when the sequence $(x_{i})$ takes the value $``0"$ at some coordinates. It is an application of Proposition~\ref{P2} and Theorem~\ref{T4}. 
\begin{cor}\label{C320}
    Let $(x_{i})\in S(\ell^{p}(X))$. Then $I:(S_{(x_{i})}, \textnormal{w}^*)\to (S_{(x_{i})}, \textnormal{w})$ is continuous at a point $(x_{i}^*)$ if and only if for each $i\in \mathbb{N}$, where $x_{i}^*\neq 0$, the functional $\frac{x_{i}^*}{\|x_{i}^*\|}\in S_{\frac{x_{i}}{\|x_{i}\|}}$ is a \textnormal{w}$^*$-\textnormal{w} PC.
\end{cor}
\begin{proof}
    Let $\Lambda=\{i\in\mathbb{N}:x_{i}=0\}$. Then we can write $\ell^{p}(X)=\ell^{p}(\Lambda, X)\bigoplus_{\ell^{p}}\ell^{p}(\Lambda^c, X)$. We may consider $(x_{i})\in \ell^{p}(\Lambda^c, X)$. From Lemma~\ref{L3}, we have that $x_{i}^*=0$ if and only if $x_{i}=0$, that is, $i\in \Lambda$. So we may also consider that the functional $(x_{i}^*)\in \ell^{q}(\Lambda^c, X^*)$. Now we apply the Theorem~\ref{T4} to the space $\ell^{p}(\Lambda^c, X)$ to get that $(x_{i}^*)\in S_{(x_{i})}^{\ell^{q}(\Lambda^c, X^*)}$ is w$^*$-w PC if and only if $\frac{x_{i}^*}{\|x_{i}^*\|}$ is w$^*$-w PC for each $i\in \Lambda^c$. Hence the conclusion follows from Proposition~\ref{P2}.
\end{proof}
Now we discuss the weak compactness of the state spaces using a similar idea of 	``convergence of nets in the sate spaces" as in Theorem~\ref{T4}.
\begin{cor}\label{C321}
     Let $(x_{i})\in S(\ell^{p}(X))$. Then $S_{(x_{i})}$ is weakly compact if and only if for each $i\in \mathbb{N}$, where $x_{i}\neq 0$, the state space $S_{\frac{x_{i}}{\|x_{i}\|}}$ is weakly compact.
\end{cor}
\begin{proof}
   Let $\Lambda=\{i\in \mathbb{N}:x_{i}=0\}$. First we assume that $\Lambda$ is an empty set, that is, $x_{i}\neq0$ for all $i\in \mathbb{N}$. Suppose $S_{(x_{i})}$ is weakly compact. For $i\in \mathbb{N}$, let $\Pi_{i}:S_{(x_{i})}\to X^*$ that sends an element of $S_{(x_{i})}$ to its $i^{th}$ coordinate. It is easy to check that each $\Pi_{i}$ is weakly continuous. Since $S_{(x_{i})}$ is weakly compact, we get that each cordinatewise state space $S_{\frac{x_{i}}{\|x_{i}\|}}$ is weakly compact in $X^*$. 

   Conversely, suppose for each $i\in \mathbb{N}$, the state space $S_{\frac{x_{i}}{\|x_{i}\|}}$ is weakly compact. Let $\{(x_{i}^*(\alpha))\}_{\alpha\in J}\subseteq S_{(x_{i})}$ be a net. Since $S_{(x_{i})}$ is w$^*$-compact, there exists a subnet, says again, $\{(x_{i}^*(\alpha))\}_{\alpha\in J}$ and an element $(x_{i}^*)\in S_{(x_{i})}$ such that $(x_{i}^*(\alpha))\to (x_{i}^*)$ in the w$^*$-topology. From Lemma~\ref{L3}, we have that $\|x_{i}^*(\alpha)\|=\|x_{i}^*\|$ for each $\alpha\in J$ and $i\in \mathbb{N}$. Therefore, we get that $\frac{x_{i}^*(\alpha)}{\|x_{i}^*(\alpha)\|}\to \frac{x_{i}^*}{\|x_{i}^*\|}$ in the w$^*$-topology. From Remark~\ref{R24}, it follows that for each $i \in \mathbb{N}$, the points $\frac{x_{i}^*(\alpha)}{\|x_{i}^*(\alpha)\|}$ and $\frac{x_{i}^*}{\|x_{i}^*\|}$ belong to the corresponding state space $S_{\frac{x_{i}}{\|x_{i}\|}}$. By the hypothesis, since each state space $S_{\frac{x_{i}}{\|x_{i}\|}}$ is weakly compact, we have that w-topology and w$^*$-topology coincides on $S_{\frac{x_{i}}{\|x_{i}\|}}$. Thus we get that $\frac{x_{i}^*(\alpha)}{\|x_{i}^*(\alpha)\|}\to \frac{x_{i}^*}{\|x_{i}^*\|}$ in the w-topology. Now we can use an argument similar to the one used in Theorem~\ref{T4} to get that the net $(x_{i}^*(\alpha))\to (x_{i}^*)$ in the w-topology. Thus, we showed that every net in $S_{(x_{i})}$ has a weakly converegent subnet. So $S_{(x_{i})}$ is weakly compact.

   If $\Lambda$ is non-empty, then we write $\ell^{p}(X)=\ell^{p}(\Lambda, X)\bigoplus_{\ell^{p}}\ell^{p}(\Lambda^c, X)$. We may consider $(x_{i})\in \ell^{p}(\Lambda^c, X)$. Hence the conclusion follows from Corollary~\ref{R2}. This completes the proof.
\end{proof}
In the following theorem, we give a similar characterization for w$^*$-$\|.\|$ PC's of the state spaces in the space $\ell^{p}(X)$.
\begin{thm}\label{T4'}
    Let $(x_{i})\in S(\ell^{p}(X))$ with $x_{i}\neq 0$ for each $i\in \mathbb{N}$. Then $I:(S_{(x_{i})}, \textnormal{w}^*)\to (S_{(x_{i})}, \|.\|)$ is continuous at a point $(x_{i}^*)$ if and only if for each $i\in \mathbb{N}$, the functional $\frac{x_{i}^*}{\|x_{i}^*\|}\in S_{\frac{x_{i}}{\|x_{i}\|}}$ is a \textnormal{w}$^*$-$\|.\|$ PC.
\end{thm}
\begin{proof}
     One can use an argument similar to the one used in Theorem~\ref{T4} to establish the necessary part.

     Conversely, suppose for each $i\in \mathbb{N}$, the functional $\frac{x_{i}^*}{\|x_{i}^*\|}\in S_{\frac{x_{i}}{\|x_{i}\|}}$ is a w$^*$-$\|.\|$ PC. Let $\{(x_{i}^*(\alpha))\}_{\alpha\in J}\subseteq S_{(x_{i})}$ be a net such that $(x_{i}^*(\alpha))\to (x_{i}^*)$ in the w$^*$-topology. From Lemma~\ref{L3}, we have that $\|x_{i}^*(\alpha)\|=\|x_{i}^*\|$ for each $i\in\mathbb{N}$ and for each $\alpha\in J$. To show that $(x_{i}^*(\alpha))\to (x_{i}^*)$ in norm, let $\epsilon>0$. Since $(\|x_{i}^*\|)\in S(\ell^{q})$, by the Hölder's inequality, we can choose $M\in \mathbb{N}$ such that $\sum_{i=M+1}^{\infty}\|x_{i}^{*}\|^{q}<\frac{\epsilon}{2^{q+1}}.$
     
     Now we claim that there exists $\alpha_{0}\in J$ such that $\sum_{i=1}^{\infty}\|(x_{i}^*(\alpha)-x_{i}^*)\|^{q}<\epsilon$ for all $\alpha\geq \alpha_{0}$. We have 
     \begin{align*}
        \sum_{i=1}^{\infty}\|(x_{i}^*(\alpha)-x_{i}^*)\|^{q}&=\sum_{i=1}^{\infty}\|x_{i}^*\|^{q}\|\left(\frac{x_{i}^*(\alpha)}{\|x_{i}^*(\alpha)\|}-\frac{x_{i}^*}{\|x_{i}^*\|}\right)\|^{q}\\
        &=\sum_{i=1}^{M}\|x_{i}^*\|^{q}\|\left(\frac{x_{i}^*(\alpha)}{\|x_{i}^*(\alpha)\|}-\frac{x_{i}^*}{\|x_{i}^*\|}\right)\|^{q}\\
        &+\sum_{i=M+1}^{\infty}\|x_{i}^*\|^{q}\|\left(\frac{x_{i}^*(\alpha)}{\|x_{i}^*(\alpha)\|}-\frac{x_{i}^*}{\|x_{i}^*\|}\right)\|^{q}.\\
\end{align*}
Since $\left(\|\frac{x_{i}^*(\alpha)}{\|x_{i}^*(\alpha)\|}-\frac{x_{i}^*}{\|x_{i}^*\|}\|\right)\leq 2$, we get that 
        \begin{align}
            \sum_{i=1}^{\infty}\|(x_{i}^*(\alpha)-x_{i}^*)\|^{q}\leq\sum_{i=1}^{M}\|x_{i}^*\|^{q}\|\left(\frac{x_{i}^*(\alpha)}{\|x_{i}^*(\alpha)\|}-\frac{x_{i}^*}{\|x_{i}^*\|}\right)\|^{q}+\frac{\epsilon}{2}.
        \end{align}
        By our assumption, $(x_{i}^*(\alpha))\to (x_{i}^*)$ in the w$^*$-topology. This implies that for each $i\in\mathbb{N}$, we have that $x_{i}^*(\alpha)\to x_{i}^*$ in the w$^*$-topology. As $\|x_{i}^*(\alpha)\|=\|x_{i}^*\|$, we get that $\frac{x_{i}^*(\alpha)}{\|x_{i}^*(\alpha)\|}\to \frac{x_{i}^*}{\|x_{i}^*\|}$ in the w$^*$-topology. By the hypothesis, we have that $\frac{x_{i}^*(\alpha)}{\|x_{i}^*(\alpha)\|}\to \frac{x_{i}^*}{\|x_{i}^*\|}$ in norm. We choose $\alpha_{0}\in J$ such that for each $1\leq i\leq M$, we have 
        \begin{align}
        \|\left(\frac{x_{i}^*(\alpha)}{\|x_{i}^*(\alpha)\|}-\frac{x_{i}^*}{\|x_{i}^*\|}\right)\|<\frac{\epsilon}{2M\|x_{i}^*\|^{q}} \text{ for all } \alpha\geq \alpha_{0}.
    \end{align}
     By Combining $(3)$ and $(4)$, we get that $\sum_{i=1}^{\infty}\|(x_{i}^*(\alpha)-x_{i}^*)\|^{q}<\epsilon$ for all $\alpha\geq \alpha_{0}$. This establishes the claim. Hence $(x_{i}^*(\alpha))\to (x_{i}^*)$ in norm. This completes the proof.
\end{proof}
Using an argument similar to the one used in Corollary~\ref{C320}, we have the following corollary.
\begin{cor}
 Let $(x_{i})\in S(\ell^{p}(X))$. Then $I:(S_{(x_{i})}, \textnormal{w}^*)\to (S_{(x_{i})}, \|.\|)$ is continuous at a point $(x_{i}^*)$ if and only if for each $i\in \mathbb{N}$, where $x_{i}^*\neq 0$, the functional $\frac{x_{i}^*}{\|x_{i}^*\|}\in S_{\frac{x_{i}}{\|x_{i}\|}}$ is a \textnormal{w}$^*$-$\|.\|$ PC.   
\end{cor}

The next corollary provides an analogue of Corollary~\ref{C321}, formulated for norm compactness.

\begin{cor}\label{C321'}
    Let $(x_{i})\in S(\ell^{p}(X))$. Then $S_{(x_{i})}$ is norm compact if and only if for each $i\in \mathbb{N}$, where $x_{i}\neq 0$, the state space $S_{\frac{x_{i}}{\|x_{i}\|}}$ is norm compact.
\end{cor}
\begin{proof}
	Without loss of generality, let us assume that $x_{i}\neq 0$ for each $i\in \mathbb{N}$. Suppose $S_{(x_{i})}$ is norm compact. One can use an argument similar to the one used in Corollary~\ref{C321} to get that each coordinatewise state space $S_{\frac{x_{i}}{\|x_{i}\|}}$ is norm compact.

	Conversely, suppose for each $i\in \mathbb{N}$, the state space $S_{\frac{x_{i}}{\|x_{i}\|}}$ is norm compact. Let $\{(x_{i}^*(\alpha))\}_{\alpha\in J}\subseteq S_{(x_{i})}$ be a net. Since $S_{(x_{i})}$ is w$^*$-compact, there exists a subnet, says again, $\{(x_{i}^*(\alpha))\}_{\alpha\in J}$ and an element $(x_{i}^*)\in S_{(x_{i})}$ such that $(x_{i}^*(\alpha))\to (x_{i}^*)$ in the w$^*$-topology. From Lemma~\ref{L3}, we have that $\|x_{i}^*(\alpha)\|=\|x_{i}^*\|$ for each $\alpha\in J$ and $i\in \mathbb{N}$. Therefore, we get that $\frac{x_{i}^*(\alpha)}{\|x_{i}^*(\alpha)\|}\to \frac{x_{i}^*}{\|x_{i}^*\|}$ in the w$^*$-topology. From Remark~\ref{R24}, it follows that for each $i \in \mathbb{N}$, the points $\frac{x_{i}^*(\alpha)}{\|x_{i}^*(\alpha)\|}$ and $\frac{x_{i}^*}{\|x_{i}^*\|}$ belong to the corresponding state space $S_{\frac{x_{i}}{\|x_{i}\|}}$. By the hypothesis, since each state space $S_{\frac{x_{i}}{\|x_{i}\|}}$ is norm compact, we have that norm topology and w$^*$-topology coincides on $S_{\frac{x_{i}}{\|x_{i}\|}}$. Thus we get that $\frac{x_{i}^*(\alpha)}{\|x_{i}^*(\alpha)\|}\to \frac{x_{i}^*}{\|x_{i}^*\|}$ in norm. Now we can use an argument similar to the one used in Theorem~\ref{T4'} to get that the net $(x_{i}^*(\alpha))\to (x_{i}^*)$ in norm. Thus, we showed that every net in $S_{(x_{i})}$ has a norm converegent subnet. So we have that $S_{(x_{i})}$ is norm compact. This completes the proof.
\end{proof}
\section{state spaces of vectors in $L^{p}(\mu, X)$}
In this section, we present results in $L^{p}(\mu, X)$ for a non-atomic probability measure $\mu$, that are analogous to those for discrete $\ell^{p}$-sums. We recall from the introduction that for any Banach space $X$, we have that $L^{q}({\mu, X^*})$ is canonically embedded in $L^{p}(\mu, X)^*$, and equality holds if and only if $X^*$ has the RNP with respect to $\mu$. Let $f \in L^{p}(\mu, X)$ with $\|f\| = 1$. We denote by $\tilde{S}_{f}$, the set $\tilde{S}_{f} = \left\{ g \in L^{q}(\mu, X^{*})_{1} : \int g(\omega)(f(\omega)) \, d\mu(\omega) = 1 \right\}.$ Note that $\tilde{S}_{f}$ may be an empty set. Throughout, we consider $1\leq p<\infty$ and $f\in L^{p}(\mu, X)$ such that $f(\omega)\neq 0$ a.e.                                    

\begin{lem}
	Let $f\in S(L^{p}(\mu,X))$ such that $\tilde{S}_{f}$ is non-empty. If $g_{1}, g_{2}\in \tilde{S}_{f}$, then $\|g_{1}({\omega})\|=\|g_{2}({\omega})\|$ a.e.
\end{lem} 
\begin{proof}
   Since $g_{1}\in \tilde{S}_{f}$, so we have that
    \begin{align*}
    1=g_{1}(f)&=\int_{\Omega} g_{1}(\omega)(f(\omega))\\
    &\leq\int_{\Omega} |g_{1}(\omega)(f(\omega))|\\
    &\leq\int_{\Omega} \|g_{1}(\omega)\|\|(f(\omega))\|.\\
    \end{align*}
    Since $(|g_{1}|)\in S(L^{q}(\mu))$ and $(|f|)\in S(L^{p}(\mu))$, by the Hölder's inequality, we have that 
    \begin{align*}
    1&\leq\int_{\Omega} \|g_{1}(\omega)\|\|(f(\omega))\|\\
    &\leq\left(\int_{\Omega}\|g_{1}(\omega)\|^{q}\right)^{\frac{1}{q}}\left(\int_{\Omega}\|f(\omega)\|^{p}\right)^{\frac{1}{p}}\\
    &=1.
    \end{align*}
   This gives that 
   \[
   \int_{\Omega} \|g_{1}(\omega)\|\|(f(\omega))\|=\left(\int_{\Omega}\|g_{1}(\omega)\|^{q}\right)^{\frac{1}{q}}\left(\int_{\Omega}\|f(\omega)\|^{p}\right)^{\frac{1}{p}}\\
    =1.
   \]
   Thus equality holds in Hölder's inequality. Consequently, the functions 
$\omega \mapsto \|g_{1}(\omega)\|^{q}$ and 
$\omega \mapsto \|f(\omega)\|^{p}$ are linearly dependent. So there exists a $\lambda\in\mathbb{C}$ such that $\|g_{1}(\omega)\|^{q}=\lambda\|f(\omega)\|^{p}$ a.e. Since $\|g_{1}\|_{q}=\|f\|_{p}=1$, we get that $\lambda=1$. Thus $\|g_{1}(\omega)\|^{q}=\|f(\omega)\|^{p}$ a.e. Similar argument when applies to $g_{2}$ gives that $\|g_{2}(\omega)\|^{q}=\|f(\omega)\|^{p}$ a.e. Combining the last two equations, we get that $\|g_{1}(\omega)\|^{q}=\|g_{2}(\omega)\|^{q}$, and hence, $\|g_{1}(\omega)\|=\|g_{2}(\omega)\|$ a.e.
\end{proof}
Let $g \in \tilde{S}_{f}$. The following remark relates the normalized coordinates of $g$ to the coordinatewise state spaces of the elements $f$.

\begin{rem}\label{R32}
Let $f\in S(L^{p}(\mu, X))$ and let $Z=\{\omega\in \Omega: f(\omega)\neq 0\}$. Clearly, $\mu(Z)>0$. From the proof of the above lemma, we observe that $g(\omega)=0$ if and only if $f(\omega)=0$. Also we note that $Re(g(\omega)(f(\omega)))=\|g(\omega)\|\|f(\omega)\|$ and $|g(\omega)(f(\omega))|=\|g(\omega)\|\|f(\omega)\|$ a.e. Therefore, $g(\omega)(f(\omega))=\|g(\omega)\|\|f(\omega)\|$ a.e. This gives that $\frac{g(\omega)}{\|g(\omega)\|}\in S_{\frac{f(\omega)}{\|f(\omega)\|}}$ for each $\omega\in Z$.
\end{rem}	

\begin{rem}\label{R33}
    Let $f\in S(L^{1}(\mu, X))$ and let $Z=\{\omega\in\Omega: \|f(\omega)\|=0\}$. Suppose that $1>\mu(Z)>0$. Now we can write $$L^{\infty}(\mu, X^*)= L^{\infty}(Z, X^*)\bigoplus_{\ell^{\infty}} L^{\infty}(Z^{c}, X^*).$$ Notice that $\|f\|=\|f\cdot \chi_{Z^{c}}\|=1$. It is easy to see that
    
    $$\tilde{S}_{f}=L^{\infty}(Z, X^*)_{1}\cross \tilde{S}_{f\cdot \chi_{Z^{c}}}.$$

   Since, the space $L^{\infty}(Z, X^*)$ is not reflexive, we get that $\tilde{S}_{f}$ is not weakly compact.
\end{rem}

We use the remark in the proof of Theorem~\ref{T36} to show that the 
weak compactness of the set $\tilde{S}_{f}$ implies that it is a singleton. Before presenting the main theorem, we recall a well-known result from \cite{Die}. It establishes a connection between the weak closure and the norm closure of a convex subset of a Banach space $X$.

\begin{thm}[Mazur]\cite[Theorem~1]{Die}\label{C31}
	Let $A$ be a convex subset of a Banach space $X$. Then $A$ is norm closed if and only if $A$ is weakly closed. 
\end{thm}

The following is an elementary fact concerning the non-atomic measures.
\begin{fact}\label{F1}
    Let $(\Omega, \Sigma, \mu)$ be a probability measure space and let $\mu$ be a non-atomic measure. Then there exists a sequence of disjoint subsets $\{Z_{n}\}_{n\in\mathbb{N}}\subseteq \Omega$ such that $\mu( \Omega\setminus\bigcup_{i=1}^{\infty} Z_{i})=0$, that is, $\mu(\bigcup_{i=1}^{n} Z_{i})\to 1$ as $n\to \infty$.
\end{fact}
 The following theorem is the main result of this section.
\begin{thm}\label{T36}
    Let $f\in S(L^{1}(\mu, X))$ such that $\tilde{S}_{f}$ is a non-empty weakly compact subset of $L^{\infty}({\mu, X^*})$. Then $\tilde{S}_{f}$ is a singleton.
\end{thm}
\begin{proof}
    Let $\tilde{S}_{f}$ be a non-empty weakly compact subset. Suppose that $\tilde{S}_{f}$ is not a singleton and let $h_{1}, h_{2}\in \tilde{S}_{f}$ be two distinct elements such that $\|h_{1}-h_{2}\|_{\infty}\geq \epsilon$ for some $\epsilon>0$. So we get that $\mu(\{\omega\in\Omega:\|h_{1}(\omega)-h_{2}(\omega)\|\geq \epsilon\})>0$. First we assume that $\mu(\{\omega\in\Omega:\|h_{1}(\omega)-h_{2}(\omega)\|\geq \epsilon\})=1$, that is, $\|h_{1}(\omega)-h_{2}(\omega)\|\geq \epsilon$ a.e. Let $A=\{\omega\in\Omega:\|h_{1}(\omega)-h_{2}(\omega)\|\geq \epsilon\}$, so $\mu(A)=1$. Also, by Remark~\ref{R32}, we have that $$h_{1}(\omega)(f(\omega))=h_{2}(\omega)(f(\omega))=\|f(\omega)\|$$ a.e. From Fact~\ref{F1}, we get a sequence of disjoint subsets $\{Z_{n}\}_{n\in\mathbb{N}}\subseteq A$ such that $\mu(\bigcup_{i=1}^{n} Z_{i})\to 1$ as $n\to \infty$. For each $n\in \mathbb{N}$, we define $h_{n}':\Omega\to X^*$ by $$h_{n}'=h_{1}\chi_{\cup_{i=1}^{n}Z_{i}}+h_{2}\chi_{\cap_{i=1}^{n}Z_{i}^{c}}.$$ It is easy to check that $h_{n}'\in L^{\infty}(\mu, X^*)_{1}$ for each $n\in \mathbb{N}$ and $h_{n}'(\omega)(f(\omega))=\|f(\omega)\|$ a.e. So we have that $h_{n}'\in \tilde{S}_{f}$ for each $n\in \mathbb{N}$. Since $\mu(\bigcup_{i=1}^{n} Z_{i})\to 1$ as $n\to \infty$, we have $\mu(\cap_{i=1}^{n}Z_{i}^{c})\to 0$ as $n\to \infty$. It is easy to see that $h_{2}\chi_{\cap_{i=1}^{n}Z_{i}^{c}}\to 0$ and $h_{1}\chi_{\cup_{i=1}^{n}Z_{i}}\to h_{1}$ in the w$^*$-topology of $L^{1}(\mu, X)^*$. Thus $h_{n}'\to h_{1}$ in the w$^*$-topology $L^{1}(\mu, X)^*$. Since $\tilde{S}_{f}$ is weakly compact in $L^{\infty}({\mu, X^*})_{1}$, so it is weakly compact in $L^{1}(\mu, X)_{1}^*$. Therefore, the w-topology and the w$^{*}$-topology coincide on $\tilde{S}_{f}$.  
It follows that $h_{n}' \to h_{1}$ in the w-topology, and hence $h_{1} \in \overline{\operatorname{co}}^{\textnormal{w}} \{h_{n}' : n \in \mathbb{N}\}.$ By Corollary~\ref{C31}, we further obtain $h_{1} \in \overline{\operatorname{co}}^{\|\cdot\|} \{h_{n}' : n \in \mathbb{N}\}.$

	Let $C$ denote an arbitrary convex combination of the elements of $\{h_{n}'\}_{n\in \mathbb{N}}$. Suppose $C=\alpha_{1}h_{n_{1}}'+\alpha_{2}h_{n_{2}}'+\dots +\alpha_{k}h_{n_{k}}'$. If we choose $m>n_{i}$ for each $1\leq i\leq k$, then $C(\omega)=h_{2}(\omega)$ for each $\omega\in Z_{m}$. Since $\mu(Z_{m})>0$, we get $\|C-h_{1}\|_{\infty}\geq \epsilon$. This is a contradiction to the fact that $h_{1}\in \overline{co}^{\|.\|}\{h_{n}':n\in\mathbb{N}\}$. Hence $\tilde{S}_{f}$ is a singleton in the case when $\mu(A)=1$. 
    
    Now suppose that $0<\mu(A)<1$. Let $\lambda=\frac{\mu|_{A}}{\mu(A)}$.  We claim that $\int_{A}\|f(\omega)\|d\lambda>0$. Suppose that $\int_{A}\|f(\omega)\|d\lambda=0$. By the definition of $\lambda$, we have $$\frac{1}{\mu(A)}\cdot\int_{A}\|f(\omega)\|d\mu=0.$$ 
     This implies $\int_{A}\|f(\omega)\|d\mu=0.$ So we have that $\|f(\omega)\|=0$ for almost all $\omega\in A$. By the hypothesis, since the set $\tilde{S}_{f}$ is weakly compact, we have from~\ref{R33} that $\mu(Z)=0$, where $Z=\{\omega\in\Omega: \|f(\omega)\|=0\}$. Consequently, $\mu(A)=0$. This gives a contradiction. It establishes the claim. Thus $\int_{A}\|f(\omega)\|d\lambda>0$.

    Let $f'=\frac{f\chi_{A}}{\|f\chi_{A}\|}$, where $\|f\chi_{A}\|=\int_{A}\|f(\omega)\|d\lambda$. We have that $f'\in S(L^{1}(A, \lambda, X))$. Now we define a map $\Theta:L^{\infty}(\Omega, \mu, X^*)\to L^{\infty}(A, \lambda, X^*)$ by $\Theta(g)=g\chi_{A}$, that is, restriction of $g$ to $A$. It easy to see that $\Theta$ is a bounded linear map from $L^{\infty}(\Omega, \mu, X^*)$ to $L^{\infty}(A, \lambda, X^*)$. We claim that $\Theta(\tilde{S}_{f})=\tilde{S}_{f'}$. Suppose $h\in \tilde{S}_{f}$, then we have that $h(\omega)(f(\omega))=\|f(\omega)\|$, $\mu$-a.e. So we get $$\Theta(h)(f\chi_{A})=\int_{A}h(\omega)(f(\omega))d\lambda=\int_{A}\|f(\omega)\|d\lambda=\|f\chi_{A}\|.$$ This implies $\Theta(h)\in \tilde{S}_{f'}$. Conversely, suppose $h\in \tilde{S}_{f'}$. Again as before, we have $h(\omega)(f(\omega))=\|f(\omega)\|$ $\lambda$-a.e on $A$. We fix some $g\in \tilde{S}_{f}$ and define $h'=h\chi_{A}+g\chi_{A^{c}}$. It is easy to see that $h'\in \tilde{S}_{f}$ and $\Theta(h')=h'\chi_{A}=h$. So we get $h\in \Theta(\tilde{S}_{f})$. This establishes the claim. As $\Theta$ is a bounded linear map and $\tilde{S}_{f}$ is weakly compact, we have that $\tilde{S}_{f'}$ is also weakly compact. Since $\Theta(h_{1}), \Theta(h_{2})\in \tilde{S}_{f'}$ and $\|\Theta(h_{1})-\Theta(h_{2})\|_{\infty}\geq \epsilon$ $\lambda$-a.e on $A$ and $\lambda(A)=1$, we can use an argument similar to the one used in the first part to obtain a contradiction to the weak compactness of $\tilde{S}_{f'}$. Therefore, in either case, we conclude that $\tilde{S}_{f}$ is a singleton. This completes the proof.
\end{proof}
The next corollary is an immediate consequence of the above theorem.
\begin{cor}
     Let $f\in S(L^{1}(\mu, X))$ such that $\tilde{S}_{f}$ is a non-empty norm compact subset of $L^{\infty}({\mu, X^*})$. Then $\tilde{S}_{f}$ is a singleton.
\end{cor}
\begin{proof}
    Since $\tilde{S}_{f}$ is norm compact and norm compactness implies weak compactness, so the conclusion follows from Theorem~\ref{T36}.
\end{proof}
If $X^{*}$ has the RNP, then $L^{1}(\mu, X)^*=L^{\infty}(\mu, X^{*})$. In this setting, the two notions of state space that we have been considering, namely $\tilde{S}_{f}$ and $S_{f}$, actually coincide. Thus, for every $f \in S(L^{1}(\mu, X))$, we have $\tilde{S}_{f} = S_{f}$, and in particular, this set is always non-empty. 
The following corollary highlights this special situation as a direct consequence of Theorem~\ref{T36}. It also provides an answer to Problem~\ref{P1}, which was left open in \cite{SDD}.

\begin{cor}\label{C37}
    Let $X$ be a Banach space such that $X^*$ has the RNP and let $f \in S(L^{1}(\mu, X))$. If the set $S_{f}$ is either weakly compact or norm compact, then $f$ is a smooth point, that is, $S_{f}$ is a singleton.

\end{cor} 
From \cite[Theorem~III.3.1]{DU}, it is well known that every separable dual space has the RNP. Let $X^*$ be a separable Banach space. Consequently, $X$ is a separable Banach space. It is known (see \cite{Ho}) that the set of all smooth points forms a $G_{\delta}$ subset of $S(X)$. Since the weak compactness or norm compactness of the state spaces coincide with the smoothness in the space $L^{1}(\mu, X)$. We arrive at the following corollary. 
 
\begin{cor}
Let $X^*, L^{1}(\mu)$ be separable spaces. Then the set of all $f\in L^{1}(\mu, X)_{1}$ such that $S_{f}$ is either weakly compact or norm compact forms a dense $G_{\delta}$ subset of $L^{1}(\mu, X)_{1}$.
\end{cor}
\begin{proof}
    Since $X^*, L^{1}(\mu)$ are both separable, we have that $L^{1}(\mu, X)$ is a separable space (see, \cite[Pg~228]{DU}) and $X^*$ has the RNP(see \cite[Pg~79]{DU} ). Now by Mazur's theorem (see \cite{Ho}), we know that the set of all smooth points of $L^{1}(\mu, X)_{1}$ is a dense $G_{\delta}$ subset of $S(L^{1}(\mu, X))$. Hence the conclusion follows from corollary~\ref{C37}.
\end{proof}
\begin{rem}\label{R37}
As noted earlier, the set $\tilde{S}_{f}$ may be an empty set. Therefore, it is a natural and interesting problem to determine the cases in which $\tilde{S}_{f}$ is non-empty. It is easy to see that $\tilde{S}_{f}$ is non-empty, when $f$ is a simple function.
\end{rem} 

\subsection{The Unit Balls $L^{\infty}(\mu, X^{*})_{1}$ vs.~$L^{1}(\mu, X)^{*}_{1}$}
Let $Y\subseteq X$ be a closed subspace and let $M\subseteq X_{1}^*$. We say that $M$ is a norming subset for $Y$ if for each $y\in Y$ we have $$\|y\|=\textnormal{sup}\{|m^*(y)|:m^*\in M\}.$$ We begin with the following general lemma, which is well known. For the sake of completeness, we include its proof.

\begin{lem}\label{L310}
    Let $Y\subseteq X$ be a dense subspace and let $M\subseteq X_{1}^*$ be a norming subset for $Y$. Then for each $x\in X$, we have  
	\begin{align}
	\|x\|=\textnormal{sup}\{|m^*(x)|:m^*\in M\}.
	\end{align}
\end{lem}
\begin{proof}
Let $x\in X$. As right hand side of the equation~(5) is always less than or equal to the left hand side, It is enough to show that $$\|x\|\leq\textnormal{sup}\{|m^*(x)|:m^*\in M\}.$$ Since $Y$ is dense in $X$, there exists a sequence of the elements of $Y$, say, $\{y_{n}\}_{n\in \mathbb{N}}$ such that $y_{n}\to x$ in norm. Consequently, we have that $\|y_{n}\|\to \|x\|$. By the hypothesis, each $y_{n}$ satisfies equation~(5). Thus $\|y_{n}\|=\textnormal{sup}\{|m^*(y_{n})|:m^*\in M\}$. Fix $\epsilon>0$ and let $K\in \mathbb{N}$ be such that $\|y_{n}-x\|<\epsilon$ for all $n\geq K$. Combining all these things together, we get that
\begin{align*}
    \|x\|&=\textnormal{lim}_{n\to \infty}\|y_n\|\\
    &=\textnormal{lim}_{n\to \infty}\textnormal{sup}\{|m^*(y_{n})|:m^*\in M\}\\
	\end{align*}
	By Triangle inequality, for each $m^*\in M$, we have 
	\begin{align*}
	|m^*(y_{n})|&\leq\|m^*\|\|(y_{n}-x)\|+|m^*(x)|\\
    &\leq \epsilon+|m^*(x)|\\
\end{align*}
 for each $n\geq K$.Thus,
	$$\|x\|=\textnormal{lim}_{n\to \infty}\textnormal{sup}\{|m^*(y_{n})|:m^*\in M\}\leq \epsilon+ \textnormal{sup}\{|m^*(x)|:m^*\in M\}$$
Since $\epsilon$ is arbitrary, it gives the conclusion. In addition, As a consequence of Hahn-Banach separation theorem, we have that the subset $M$ is w$^*$-dense in $X_{1}^*$.
\end{proof}
Note that $L^{\infty}(\mu, X^*)$ is isometrically isomorphic to a closed subspace of the dual space $L^{1}(\mu, X)^*$. We shall use this fact, together with Lemma~\ref{L310}, to establish the following theorem.
\begin{thm}\label{T312}
    Let the set of all \textnormal{w}$^*$-\textnormal{w} PC's in the dual unit ball $L^{1}(\mu, X)_{1}^*$ be a weakly dense subset of $L^{1}(\mu, X)_{1}^*$. Then $X^*$ has the RNP.
\end{thm}
\begin{proof}
    Let $M=L^{\infty}(\mu, X^*)_{1}$ and let $Y=\{s\in L^{1}(\mu, X):s \textnormal{ is a simple function }\}$. Now we claim that $M$ is a norming set for the subspace $Y$. Suppose $s\in L^{1}(\mu, X)$ is a simple function, say, $s=\sum_{i=1}^{k}x_{i}\chi_{E_{i}}$ for some $x_{i}\in X$ and measurable subsets $E_{i}$'s. For $1\leq i\leq k$, we choose $x_{i}^*\in X_{1}^*$ such that $x_{i}^*(x_{i})=\|x_{i}\|$. We consider $g=\sum_{i=1}^{k}x_{i}^*\chi_{E_{i}}.$ It is straightforward to check that $g\in L^{\infty}(\mu, X^*)_{1}$ and $g(s)=\|s\|$. Thus every element of $Y$ attain its norm on $M$. This establishes the claim. Since simple functions are dense in $L^{1}(\mu, X)$, by applying Lemma~\ref{L310} to our choice of $Y$ and $M$, we get that the dual unit ball $L^{\infty}(\mu, X^*)_{1}$ is w$^*$-dense in $L^{1}(\mu, X)_{1}^*$. Let $\Phi\in L^{1}(\mu, X)_{1}^*$ be a \textnormal{w}$^*$-\textnormal{w} PC. 
Since $L^{\infty}(\mu, X^*)_{1}$ is w$^*$-dense in $L^{1}(\mu, X)_{1}^*$, 
there exists a net $\{\Phi_{\alpha}\}_{\alpha\in J}\subseteq 
L^{\infty}(\mu, X^*)_{1}$ such that $\Phi_{\alpha}\to \Phi$ in the 
w$^*$-topology, and hence in the w-topology. 
It follows that $\Phi\in L^{\infty}(\mu, X^*)_{1}$, since it is weakly closed (see, Theorem~\ref{C31}). Thus \textnormal{w}$^*$-\textnormal{w} PC's of 
$L^{1}(\mu, X)_{1}^*$ must belong to the unit ball 
$L^{\infty}(\mu, X^*)_{1}$. By the hypothesis, we have that the set of all \textnormal{w}$^*$-\textnormal{w} PC's in the dual unit ball $L^{1}(\mu, X)_{1}^*$ is a weakly dense subset of $L^{1}(\mu, X)_{1}^*$. This implies that  $$L^{\infty}(\mu, X^*)_{1}=L^{1}(\mu, X)_{1}^*.$$ By a classical result (see \cite[Theorem~IV.1.1]{DU}) concerning the duality of $L^{1}(\mu, X)$, we get that $X^*$ has the RNP. This completes the proof.
\end{proof}

 We denote the canonical embedding of $X$ in to its bidual $X^{**}$ by identifying $X$ with its image under the map $J_{X}:X\to X^{**}$, where $J_{X}$ is the canonical embedding. We recall from the introduction that a point $x\in S(X)$ is said to be smooth if $S_{x}$ is a singleton. In \cite{SF}, the authors introduced another level of smoothness for smooth points, namely, \textbf{very smooth points}. A smooth point $x \in S(X)$ with $S_{x} = \{x^*\}$ is said to be \textbf{very smooth} if its canonical image in $X^{**}$ under the mapping $J_X$, that is, $J_X(x) \in X_{1}^{**}$, is itself a smooth point.
 It is easy to see that $x$ is a very smooth point in $S(X)$ if and only if $x^*\in S(X^*)$ has a unique norm preserving extension from $X$ to $X^{**}$. However, it is shown in \cite{GG} that a functional $x^*\in S(X^*)$ has a unique norm preserving extension from $X$ to $X^{**}$ if and only if it is a point of continuity of the map $I:(X_{1}^*, \textnormal{w}^*)\to (X_{1}^*, \textnormal{w})$. The following proposition addresses a non-trivial case in which the set $\tilde{S}_{f}$ is non-empty.

\begin{prop}
    Let $f\in L^{1}(\mu, X)_{1}$ be a very smooth point. Then it attains its norm uniquely on $L^{\infty}(\mu, X^*)_{1}$, that is, $\tilde{S}_{f}$ is a singleton.
\end{prop}
\begin{proof}
    Let $f\in L^{1}(\mu, X)_{1}$ be a very smooth point and let $S_{f}=\{\Lambda\}$. In this case, we have that $\Lambda\in L^{1}(\mu, X)_{1}^*$ is a w$^*$-w PC. Since $L^{\infty}(\mu, X^*)_{1}$ is w$^*$-dense in $L^{1}(\mu, X)_{1}^*$, we get that $\Lambda\in L^{\infty}(\mu, X^*)_{1}$. This completes the proof.
\end{proof}
    We recall from \cite{SF} that a Banach space $X$ is said to \textbf{Hahn-Banach smooth} if in $X^{***}$, $\|x^{*}+x^{\perp}\|=\|x^*\|=1$ implies that $x^{\perp}=0$ in other words $x^*\in X^{***}$ is the unique Hahn-Banach extension of $x^*|_{X}$. We now state the following known result, which establishes a connection between separable Banach spaces and the separability of their duals. We provide a short proof for the sake of completeness.
\begin{prop}\label{C313}
    Let $X$ be a Hahn-Banach smooth space. If $X$ is separable, Then $X^*$ is a separable space.
\end{prop}
\begin{proof} 
    Since $X$ is separable and the dual unit ball $X_{1}^*$ is w$^*$-compact, by \cite[Theorem~3.16]{RW}, we get that $(X_{1}^*, \textnormal{w}^*)$ is a compact metric space. It is separable. We can choose a countable w$^*$-dense set $\{x_{n}^*\}_{n\in\mathbb{N}}\subseteq S(X^*)$ such that $X_{1}^*=\overline{co}^{\textnormal{w}^*}\{x_{n}^*:n\in \mathbb{N}\}$. Since $X$ is a Hahn-Banach smooth space, we have that each point on $S(X^*)$ is a w$^*$-w PC. By Mazur's theorem (see Theorem~\ref{C31}), we get $$X_{1}^*=\overline{co}^{\textnormal{w}}\{x_{n}^*:n\in \mathbb{N}\}=\overline{co}^{\|.\|}\{x_{n}^*:n\in \mathbb{N}\}.$$ Thus $X^*$ ia a separable space.
\end{proof}

The following theorem is an interesting application of Theorem~\ref{T36} and Proposition~\ref{C313}.  

\begin{thm}
   Let $X$ be a Hahn-Banach smooth separable space and let $f\in S(L^{1}(\mu, X))$ be such that $S_{f}$ is weakly compact. Then $\frac{f(\omega)}{\|f(\omega)\|}\in S(X)$ is a very smooth point a.e.
\end{thm}
\begin{proof}
    Since $X$ is a Hahn-Banach smooth separable space, by  Corollary~\ref{C313}, we have that $X^*$ is a separable space. So it has the RNP (see \cite[Theorem~III.3.1]{DU}). Therefore, we have that $\tilde{S}_{f}=S_{f}$. By Theorem~\ref{T36}, we get that $S_{f}$ is a singleton, that is, $f\in S(L^{1}(\mu, X))$ is a smooth point. It follows that $\frac{f(\omega)}{\|f(\omega)\|}\in S(X)$ is a smooth point a.e (see \cite[Theorem~2.3]{SDD}). Let $S_{f}=\{g\}$. By the hypothesis, since $X$ is Hahn-Banach smooth, we get that $g(\omega)\in S(X^*)$ is a w$^*$-w PC and $g(\omega)(f(\omega))=\|f(\omega)\|$ a.e. This implies that $\frac{f(\omega)}{\|f(\omega)\|}\in S(X)$ is a very smooth point a.e.
\end{proof}
The following example presents a situation in which it is shown that, for a function $f \in S(L^{1}(\mu, X))$,  
the condition that $\frac{f(\omega)}{\|f(\omega)\|}\in S(X)$ is very smooth a.e does not necessarily imply that $f$ itself is very smooth.

\begin{ex}
Let $X$ be a reflexive and smooth space and let $f\in L^{1}(\mu, X)_{1}$ be a smooth point with $S_{f}=\{g\}$. Since $X$ is a smooth space, we have that $\frac{f(\omega)}{\|f(\omega)\|}$ is smooth in $X$ for almost all $\omega\in \Omega$ with $S_{\frac{f(\omega)}{\|f(\omega)\|}}=\{g(\omega)\}$. As $X$ is reflexive, we have that $g(\omega)$ is a \textnormal{w}$^*$-\textnormal{w} PC of $S(X^*)$ a.e. Consequently, the point $\frac{f(\omega)}{\|f(\omega)\|}\in S(X)$ is very smooth a.e. By Levin's theorem (see \cite{LV}), we have the canonical decomposition $$L^{\infty}(\mu, X^*)^*=L^{1}(\mu, X)\bigoplus_{1}M$$ Here, $L^{1}(\mu, X)$ is canonically embedded into its bidual $L^{\infty}(\mu, X^*)^{*}$, and 
$M \subseteq L^{\infty}(\mu, X^*)^{*}$ denotes a closed subspace.
 This implies that $f\in L^{1}(\mu, X)_{1}$ is not a very smooth point (see, \cite[Proposition~23]{TS}). 

\end{ex}
\subsection{Open Problems: } In what follows, we present some open problems.

\begin{problem}
In Theorems~\ref{T4} and \ref{T4'}, we studied the \textnormal{w}$^{*}$-\textnormal{w} and \textnormal{w}$^{*}$-$\|.\|$ points of continuity of the state spaces.  
However, a characterization of \textnormal{w}$^{*}$-\textnormal{w} points of continuity of the dual unit ball $\ell^{q}(X^{*})_{1}$ is not known.  
In particular, even for a norm-attaining functional $(x_{i}^{*}) \in \ell^{q}(X^{*})$, the conclusion remains open.
\end{problem}

\begin{problem}
 In Theorem~\ref{T36}, we have presented a characterization of weakly compact state spaces in the space $L^{1}(\mu, X)$. A similar result is not known for the space $L^{p}(\mu, X)$, where $1<p<\infty$.
\end{problem}
 Motivated by our last theorem, namely, Theorem~\ref{T312}, we can ask the following question.
\begin{problem}
	If all the elements in $L^{p}(\mu, X)$ attain its norm on $L^{q}(\mu, X^*)_{1}$, then $X^*$ has the RNP?
\end{problem}

\vspace{2mm}
\noindent\small{\bf Conflict of interest:} The author declares that there is no conflict of interest.

\subsection*{Acknowledgements}
The author would like to thank Prof.~T.~S.~S.~R.~K.~Rao and Dr.~Priyanka Grover for their invaluable guidance, constant support, and inspiring mentorship throughout this work.  
The author is also grateful for their encouragement during the preparation of the thesis, of which this paper forms a part.  
The author would also like to thank the referee for the careful reading of the manuscript and for valuable suggestions that improved the presentation of the paper.  
Sincere thanks are due to the department for providing a stimulating academic environment that greatly facilitated this research.

\end{document}